\RequirePackage{fix-cm}
\RequirePackage{amsmath}

\documentclass[numbook,envcountsame,smallcondensed]{svjour3}           
\smartqed

\usepackage{mathptmx}
\usepackage{amsfonts}
\usepackage{mathtools}
\mathtoolsset{showonlyrefs,showmanualtags}

\usepackage[hyphens]{url}
\usepackage{hyperref}
\hypersetup{breaklinks=true}

\newcommand\End{\mathrm{End}}
\newcommand\Ric{\mathrm{Ric}} \newcommand\Hess{\mathrm{Hess}}
 \newcommand\II{\mathrm{II}}
\newcommand\Sect{\mathrm{Sect}}

 \newcommand\E{\mathbb{E}}
 
 \newcommand\R{\mathbb{R}}
\renewcommand\P{\mathbb{P}}

\newcommand\vd{\mathrm{d}}
\newcommand\e{\operatorname{e}}
\newcommand\id{\mathrm{id}}

\newcommand\blank{{\kern.8pt\displaystyle\cdot\kern.8pt}}
\newcommand\blankdown{{\!\displaystyle.\kern.8pt}}

\def\r{\right}
\def\l{\left}

\journalname{Analysis and Mathematical Physics,}

\begin{document}

\title{Uniform gradient estimates on manifolds with a boundary \\ and applications}

\titlerunning{Uniform gradient estimates on manifolds with a boundary}   

\author{Li-Juan Cheng         \and
        Anton Thalmaier	      \and
        James Thompson
}

\authorrunning{L.-J. Cheng, A. Thalmaier and J. Thompson} 

\institute{Li-Juan  Cheng \at
              Mathematics Research Unit, 
              University of Luxembourg, Campus Belval, 
              4364 Esch-sur-Alzette, 
              Luxembourg\\
              \email{lijuan.cheng@uni.lu}
\and
           Anton Thalmaier  \at
           Mathematics Research Unit, 
              University of Luxembourg, Campus Belval, 
              4364 Esch-sur-Alzette, 
              Luxembourg\\
              \email{anton.thalmaier@uni.lu}
\and
           James Thompson \at Mathematics Research Unit, 
              University of Luxembourg, Campus Belval, 
              4364 Esch-sur-Alzette, 
              Luxembourg\\
                \email{james.thompson@uni.lu}
}


\date{Date: March 23, 2018}

\maketitle

\begin{abstract}
We revisit the problem of obtaining uniform gradient estimates for Dirichlet and Neumann heat semigroups on Riemannian manifolds with boundary. As applications, we obtain isoperimetric inequalities, using Ledoux's argument, and uniform quantitative gradient estimates, firstly for $C^2_b$ functions with boundary conditions and then for the unit spectral projection operators of Dirichlet and Neumann Laplacians.
\end{abstract}

\section{Introduction}

Suppose $M$ is a complete and connected Riemannian manifold of dimension $d$. Denote by $\rho$ the Riemannian distance function, by $\nabla$ the Levi-Civita connection, by $\Delta$ the Laplace-Beltrami operator and for a smooth vector field $Z$ consider the elliptic operator $L := \Delta +Z$.

If $M$ is without a boundary, then it is easy to show that if the Bakry-\'Emery Ricci tensor ${\Ric}^Z:=\Ric-\nabla Z$ is bounded below then the semigroup of the $\frac{1}{2}L$-diffusion satisfies a uniform gradient estimate on the type displayed below in Theorem \ref{global-esti}. See also \cite{BL96,Elworthy}. If $M$ has a boundary, estimates of this type can be extended to the Neumann semigroup (the semigoup of the reflecting $\frac{1}{2}L$-diffusion process), as in \cite{Wbook2}, or to the Dirichlet semigroup (the semigroup of the diffusion killed on the boundary), as in \cite{Wang04}. In the latter, Wang used coupling methods to obtain the estimate. Isoperimetric inequalities were derived as a consequence. The purpose of the present paper is to revisit these problems, considering uniform gradient estimates for both the Neumann and Dirichlet semigroups on manifolds with boundary, and to present applications.

Our gradient estimates for the Dirichlet and Neumann semigroups are given below by Theorems \ref{thm-gradient-Dirichlet} and \ref{Gradient-Neumann}, respectively. For the Neumann case, Wang \cite{Wang05,Wbook2} established a derivative estimate, as mentioned above, using a Bismut formula and conformal change of metric. Theorem \ref{Gradient-Neumann} is a consequence of a strengthening of Wang's result, which we presented in our recent article \cite{Ch-Tha-Tho:2017b}. To prove Theorem \ref{thm-gradient-Dirichlet}, we build upon the recent work of Arnaudon, Thalmaier and Wang \cite{Ar-Tha-Wa:2017}, in which explicit uniform two-sided gradient estimates for Dirichlet eigenfunctions were proved using a probabilistic method, also based on Bismut's formula.

Uniform gradient estimates for the semigroup can, of course, be applied to give estimates for the eigenfunctions, as explained at the end of Subsection \ref{ss:gradnobound}. This leads to an application of our main results, namely the derivation of uniform estimates for the gradient of unit spectral projection operators. The study of these objects, which are constructed from eigenfunctions, has a long history \cite{SS07,Sogge02,Sogge88,Xu11,Xu09}. In \cite{Xu09,Xu11}, Xu's method used local gradient estimates twice; for points far from the boundary and for points close the boundary. This approach can be simplified; see Subsection \ref{ss:spec} for our main results in this direction.

In some sense, this paper is a continuation of \cite{Ch-Tho-Tha:2017a}. In that paper, we proved quantitative local $C^1$-estimates for $C^2$ functions on manifolds without a boundary, with extensions given to differential forms. That developed the recent work of G\"{u}neysu and Pigola \cite{GP2017}, whose argument was based on Taylor expansion. Our approach, on the other hand, uses stochastic analysis. In this article we consider global curvature bounds, on manifolds with a boundary, for which it is possible to obtain uniform versions of the estimates we proved in \cite{Ch-Tho-Tha:2017a}. Our main results in this direction are Theorems \ref{thm:uniformests}, \ref{c1-est-d} and \ref{c1-est-n}. Our results on the spectral projectors are based on these theorems.

Another remarkable application of uniform gradient estimates, as mentioned above, is that they can be used to obtain isoperimetric inequalities. For the case $Z=0$, Buser \cite{Buser82} obtained a lower bound on Cheeger's isoperimetric constant using Poincar\'{e} inequalities; his proof was further simplified by Ledoux \cite{Ledoux94} using a uniform gradient estimate of the type introduced above. Wang \cite{Wang00} then applied Ledoux's argument to obtain lower bounds for various isoperimetric constants using Poincar\'{e}-Sobolev inequalities. Using our new uniform gradient estimates, for the Neumann and Dirichlet semigroups, we consequently obtain isoperimetric inequalities with better constants, given in Subsection \ref{ss:isoineq}.

So the paper is organized as follows. In Section \ref{sec:gradests}, we present the uniform gradient estimates for the heat semigroup on a manifold without a boundary and the Dirichlet and Neumann heat semigroups on manifolds with a boundary, followed in Subsection \ref{ss:isoineq} by the isoperimetric inequalities. In Section \ref{sec:quant}, we present the uniform $C^1$-estimates in an analogous order, followed in Subsection \ref{ss:spec} by our uniform upper bounds for the gradient of the unit spectral projection operators.

\section{Gradient estimates for diffusion semigroups}\label{sec:gradests}

In this section we first consider manifolds without boundary. We then consider Dirichlet and Neumann boundaries, respectively. Denoting by $\Ric$ the Ricci curvature tensor, note that by ${\Ric}^Z \geq K_Z$ we shall mean
\begin{equation}
{\rm Ric}^Z(X,X):=(\Ric-\langle \nabla_{\blankdown}Z, \blank \rangle)(X,X)\geq K_Z|X|^2,\quad X\in TM,
\end{equation}
supposing always that $K_Z$ is a constant. We denote by $\mathcal{B}_b(M)$ the set of all bounded measurable functions on $M$.

\subsection{No boundary}\label{ss:gradnobound}

In this subsection and the next, we denote by $X_t$ an $\frac{1}{2}L$-diffusion on $M$, defined on some maximal random time interval, and denote by $//_t$ the associated stochastic parallel transport and by $B_t$ the martingale part of the antidevelopment. In particular, if $X_0=x$ for some $x \in M$ then $B_t$ is a Brownian motion on the tangent space $T_xM$ starting at the origin. Denote by $\mathcal{Q}_s$ the $\End(T_xM)$-valued solution to the ordinary differential equation
\begin{equation}\label{eq:Qdefn}
\frac{d}{ds} \mathcal{Q}_s  = -\frac12\Ric^Z_{//_s}\mathcal{Q}_s
\end{equation}
along the paths of $X_t$ with $\mathcal{Q}_0 = \id_{T_xM}$ and $$\Ric^Z_{//_s} := //_s^{-1} \Ric^Z //_s.$$ The composition $\mathcal{Q}_s//_s^{-1}$ is called the damped parallel transport from $T_{X_s}M$ to $T_xM$. The following estimate, for the associated minimal semigroup $P_t$, is well known:

\begin{theorem}\label{global-esti}
Suppose $\Ric^Z\geq K_Z$ for some constant $K_Z$. Then for all $u\in \mathcal{B}_b(M)$ we have
\begin{align*}
\|dP_tu\|_{\infty}\leq \sqrt{\frac2\pi}\l(\frac{K_Z}{\e^{K_Zt}-1}\r)^{1/2}\|u\|_{\infty}
\end{align*}
for all $t>0$.
\end{theorem}

\begin{proof}
Suppose $h$ is a bounded adapted process with paths in the Cameron-Martin space $L^{1,2}([0,t];\R)$ such that $h(0)=1$ and $h(t) = 0$. Then, according to Bismut's formula \cite{bismut,Thalmaier97}, we have
\begin{equation}
(dP_tu)_x = -\E^x \left[ u(X_{t}) \int_0^{t} \langle \mathcal{Q}_r h'(r),dB_r\rangle\right].
\end{equation}
Letting $\sigma_t:=\E^x[\int_0^t|h'(s)|^2\|\mathcal{Q}_s\|^2\vd s]^{1/2}$ it follows that
\begin{align}\label{esti-martingale}
\E^x\l[\l|\int_0^{t} \langle \mathcal{Q}_s h'(s),dB_s\rangle\r|\r]&\leq \frac{2}{\sqrt{2\pi} \sigma_t}
\int_0^{\infty}s\exp{\l(-\frac{s^2}{2\sigma_t^2}\r)}\vd s=\sqrt{\frac{2}{\pi}}\sigma_t.
\end{align}
Taking
\begin{equation}\label{eq:choiceofh}
h(s)=\frac{\e^{K_Zt}-\e^{K_Zs}}{\e^{K_Zt}-1}
\end{equation}
we thus have
\begin{equation}
|d P_tu|(x)\leq \sqrt{\frac2\pi}\l(\int_0^th'(s)^2\e^{-K_Zs} \vd s\r)^{1/2} \|u\|_{\infty} = \sqrt{\frac2\pi}\l(\frac{K_Z}{\e^{K_Zt}-1}\r)^{1/2}\|u\|_{\infty}
\end{equation}
as required. \qed
\end{proof}

\begin{remark}\label{rem:sinzerot}
In general, if one obtains $\|dP_tu\|_{\infty}\leq \alpha(t)\|u\|_{\infty}$ then in fact
\begin{equation}
\|dP_tu\|_{\infty}\leq \alpha(s)\|P_{t-s}u\|_{\infty}\leq \alpha(s)\|u\|_{\infty}
\end{equation}
for all $s \in (0,t]$. In Theorem \ref{global-esti}, note that the coefficient is decreasing in $t$ so the upper bound is, in terms of this remark, the best choice. This observation will, however, prove quite useful later in the article.
\end{remark}

Theorem \ref{global-esti} implies a uniform gradient estimate for eigenfunctions. If $u$ is an eigenfunction of $-L$ with eigenvalue $\lambda>0$, in other words $Lu=-\lambda u$, it follows that $P_tu=\e^{-{\lambda t}/{2} }u$ and so, by Theorem \ref{global-esti}, we have
\begin{align*}
\frac{\|d u\|_{\infty}}{\|u\|_{\infty}} \leq \sqrt{\frac{2}{\pi}} \inf_{ t>0}\l(\frac{K_Z^-\e^{\lambda t}}{1-\e^{-K_Z^-t}}\r)^{1/2}= \sqrt{\frac{2}{\pi}}(\lambda+K_Z^-)^{1/2} \l(\frac{\lambda}{\lambda+K_Z^-}\r)^{\lambda/(2 K_Z^-)}.
\end{align*}
For eigenfunctions with Dirichlet boundary conditions, we direct the reader towards the recent article \cite{Ar-Tha-Wa:2017}; see also the next subsection.

\subsection{Dirichlet boundary}

Next, we consider manifolds with boundary $\partial M$, first considering the case in which the diffusion is killed on the boundary. For this we define the stopping time $\tau=\inf\{t: X_t\in \partial M\}$ and define the Dirichlet heat semigroup, acting on bounded measurable functions $u$, by $$P^D_tu(x)=\mathbb{E}^x\left[ 1_{\{t<\tau\}}u(X_t) \right].$$ We denote by $H_{\partial M}$ the mean curvature of the boundary ($H_{\partial M} \geq 0$ implies the boundary is mean-convex).

\begin{theorem}\label{thm-gradient-Dirichlet}
Assume $Z$ is bounded and suppose there exist constants $K_Z, K_0$ and $\theta$ such that $\Ric^Z\geq K_Z$, $\Ric\geq K_0$ and $H_{\partial M}\geq \theta$. Define
\begin{equation}\label{eq:alpha0defn}
\alpha_0:=\frac{1}{2}\l(\max\l\{\theta^{-}, \sqrt{(d-1)K_0^{-}}\r\}+\|Z\|_{\infty}\r)
\end{equation}
and set
\begin{align*}
C(s):=\sqrt{\frac{2}{\pi}}+\sqrt{s}\,\alpha_0\min \l\{ 2,1+\alpha_0\sqrt{\frac{s}{2\pi}}\r\}.
\end{align*}
Then for $u\in \mathcal{B}_b(M)$ and $t>0$ we have
\begin{align*}
\|dP^D_tu\|_{\infty}\leq \frac{\exp\left(\frac12K_Z^{-}s\right)}{\sqrt{s}} \l(C(s)+\frac{1}{4C(s)}\r) \|u\|_{\infty}
\end{align*}
for all $0<s\leq t$.
\end{theorem}

\begin{proof}
Suppose that $h$ is a bounded adapted process with paths belonging to the Cameron-Martin space $L^{1,2}([0,t];[0,1])$. Since $u_s:=P^D_{t-s}u$ is a solution to the backwards diffusion equation on $[0,t] \times M$, it follows, by It\^{o}'s formula and the Weitzenb\"{o}ck formula, that
\begin{equation}
du_{s}(//_s \mathcal{Q}_s h(s)) - u_{s}(X_s)\int_0^s \langle \mathcal{Q}_r \dot{h}(r),dB_r\rangle
\end{equation}
is a local martingale, where $\mathcal{Q}$ is defined by \eqref{eq:Qdefn}. If in addition $h(0) = 1$ and $h(t) = 0$, then evaluating at times $0$ and $s = t\wedge \tau$, taking expectations and using the initial and boundary conditions we obtain
\begin{align}
(dP^D_tu)_x =\text{ }& \E^x\l[1_{\lbrace t > \tau \rbrace} dP^D_{t-\tau}u(//_{\tau} \mathcal{Q}_{\tau} h(\tau))\r]\\
&-\E^x\left[1_{\lbrace t < \tau \rbrace}  u(X_{t}) \int_0^{t\wedge \tau} \langle \mathcal{Q}_r \dot{h}(r),dB_r\rangle\right].\label{eq:bisone}
\end{align}
Consequently
\begin{align*}
|dP^D_tu|(x)\leq \text{ }& \|u\|_{\infty}\e^{K_Z^{-}t/2} \E^x\l[\int_0^t\dot{h}_s^{\,2} \,\vd s\r]^{1/2} \sqrt{\P\lbrace t < \tau \rbrace}\\
&+\e^{K_Z^{-}t/2}\E\left[1_{\{\tau<t\}}\,h(\tau)\,|d P^{D}_{t-\tau}u|(X_{\tau})\right]
\end{align*}
and it remains to estimate the expectation on the right-hand side (the one involving the stopping time). For this, we follow the approach of \cite{Ar-Tha-Wa:2017}, fixing $y\in \partial M$ and letting $y^{\varepsilon}:=\exp_y(\varepsilon N)$ where $N$ is the inward pointing unit normal vector field on $\partial M$. Since $\psi(s,x):=\P\{\tau>s\}$ and $u_s=P^D_{t-s}u$ vanish on $\partial M$ when $s\in [0,t)$ it follows that
\begin{align*}
|d u_{s}|(x)=|Nu_s(x)|=\lim_{\epsilon \downarrow 0} \frac{|u_s(y^{\epsilon})|}{\epsilon}, \quad |d\psi(s,\blank)|(x)=\lim_{\epsilon \downarrow 0} \frac{|\psi(s,y^{\epsilon})|}{\epsilon}.
\end{align*}
Now letting $X_t^{\epsilon}$ be a $\frac 12L$-diffusion process starting at $y^{\epsilon}$, with first hitting time $\tau^{\epsilon}$ to the boundary $\partial M$, we have
\begin{align*}
|d u_s|(y)=\lim_{\epsilon \downarrow 0} \frac{1}{\epsilon} \left|\E[1_{\{t-s<\tau^{\epsilon}\}}u(X_{t-s}^{\epsilon})]\right| \leq \|u\|_{\infty}\,|d \psi(t-s, y)|.
\end{align*}
Thus
\begin{align*}
|dP^D_tu|(x)\leq\text{ }& \|u\|_{\infty}\e^{K_Z^{-}t/2}\E^x\l[\int_0^t \dot{h}_s^{\,2}\, \vd s\r]^{1/2}\sqrt{\P\lbrace t < \tau \rbrace}\\
&+\e^{K_Z^{-}t/2}\|u\|_{\infty}\,\E\left[1_{\{\tau<t\}}h_{\tau}|d \psi(t-\tau, \blank)(X_{\tau})|\right].
\end{align*}
It has been proved in \cite{Ar-Tha-Wa:2017} that
\begin{align*}
|d\psi(s,\blank)|(y)\leq \sqrt{\frac {2}{\pi s}} +\min\l\{2\alpha_0,\alpha_0+\alpha_0^2\sqrt{\frac{s}{2\pi}}\r\}
\end{align*}
for each $y\in \partial M$, so consequently
\begin{align*}
&|d P^D_tu|(x)\\
&\leq \|u\|_{\infty}\e^{K_Z^{-}t/2}\left(\E^x\l[\int_0^t \dot{h}_s^{\,2}\, \vd s\r]^{1/2}\sqrt{\P\lbrace t < \tau \rbrace}\right.\\
&\quad+\left.\E^x\l[1_{\{\tau<t\}}h_{\tau}\l( \sqrt{\frac {2}{\pi (t-\tau)}} +\min\l\{2\alpha_0,\alpha_0+\alpha_0^2\sqrt{\frac{t-\tau}{2\pi}}\r\}\r)\r]\right).
\end{align*}
Choosing $h_s=\frac{t-s}{t}$ yields the estimate
\begin{align*}
&|dP^D_tu|(x)\\
&\leq \|u\|_{\infty}\frac{\exp\left(K_Z^{-}t/2\right)}{\sqrt{ t}}\max_{\epsilon\in [0,1]}\l\{\sqrt{1-\epsilon}+\epsilon\l(\sqrt{\frac{2}{\pi}}+ \sqrt{t}\min\l\{2\alpha_0, \alpha_0+\alpha_0^2\sqrt{\frac{t}{2\pi}}\r\}\r)\r\}\\
&=\frac{\e^{K_Z^{-}t/2}}{\sqrt{ t}} \l(C(t)+\frac{1}{4C(t)}\r)\|u\|_{\infty}
\end{align*}
from which the result follows, by Remark \ref{rem:sinzerot}. \qed
\end{proof}

\begin{remark}If $u$ is a Dirichlet eigenfunction of $-L$ with eigenvalue $\lambda$, then Theorem \ref{thm-gradient-Dirichlet} implies
\begin{align*}
\frac{\|du\|_{\infty}}{\|u\|_{\infty}}\leq \inf_{t\geq 0} \frac{\exp\left(\frac12(\lambda+K_Z^-) t\right)}{\sqrt{t}} \l(C(t)+\frac{1}{4C(t)}\r)
\end{align*}
for all $t>0$ and therefore, taking $t=(\lambda+K_Z^-)^{-1}$, we have
\begin{align*}
\frac{\|du\|_{\infty}}{\|u\|_{\infty}}\leq \sqrt{\e (\lambda+K_Z^-)}\l(C\l(\frac{1}{\lambda +K_Z^-}\r)+\frac{\lambda+K_Z^-}{4C}\r).
\end{align*}
\end{remark}

\subsection{Neumann boundary}

The Neumann heat semigroup, acting on bounded measurable functions, is defined by $$P^N_tu(x) = \E^x[ u(\tilde{X}_t)],$$ where now $\tilde{X}_t$ denotes a $\frac{1}{2}L$-diffusion reflected on $\partial M$. In particular $N P^N_tu = 0$ for $t>0$, so $P^N_tu$ solves the diffusion equation with Neumann boundary conditions \cite[Section~3.1]{Wbook2}.

Denoting by $N$ the inward pointing unit normal vector, we define the second fundamental form $\rm II$ of the boundary $\partial M$ by $${\rm II}(X,Y)=-\l<\nabla_XN,Y\r>, \quad X,Y\in T_y\partial M, \quad y \in \partial M$$ where $T_y\partial M$ denotes the tangent space of $\partial M$ at $y$. In order to study non-convex boundaries with Neumann boundary conditions, we will perform a conformal change of metric such that the boundary is convex under the new metric. In particular, if we denote by $g$ the original metric, we will use the fact that if $$\mathcal{D}:=\{\phi\in C_b^2(M): \inf \phi = 1,\ {\rm II}\geq - N\log \phi \}$$ and $\phi\in \mathcal{D}$ then the boundary $\partial M$ is convex under the metric $\phi^{-2}g$.

\begin{theorem}\label{Gradient-Neumann}
If there exist $\phi\in \mathcal{D}$ and a constant $K_{\phi}$ such that
\begin{align}\label{K-phi}
{\rm Ric}^Z+2L \log \phi-2|\nabla \log \phi|^2 \geq K_{\phi}
\end{align}
then for $u\in \mathcal{B}_b(M)$ we have
\begin{align*}
|dP^N_tu|(x)\leq &\sqrt{\frac{2}{\pi}}  \l(\frac{K_{\phi}}{\e^{K_{\phi}t}-1}\r)^{1/2}\|\phi\|_{\infty}\|u\|_{\infty}
\end{align*}
for all $t>0$.
\end{theorem}

\begin{proof}
Since there exist $\phi\in \mathcal{D}$ and $K_{\phi} \in \mathbb{R}$ such that \eqref{K-phi} holds, it follows that $$\frac{1}{2}\Ric^Z\geq \frac{1}{2} K_{\phi}-L\log \phi +|\nabla \log \phi|^2=\frac{1}{2}K_{\phi}+\frac{1}{2}\phi^2 L\phi^{-2}\  \ \mbox{and}\quad\II \geq -N\log \phi.$$ Therefore, by our main results in \cite{Ch-Tha-Tho:2017b} and \cite[Theorem 3.2.1]{Wbook2}, it follows that there exists an adapted $\End(T_xM)$-valued process $\lbrace \tilde{\mathcal{Q}}_s \rbrace_{s\in [0,t]}$ such that
\begin{align}\label{Q-upper}
\|\tilde{\mathcal{Q}}_s\|\leq \exp\l(-\frac{1}{2}K_{\phi} s-\frac{1}{2}\int_0^s\phi^2 L\phi^{-2}(\tilde{X}_r)\,\vd r+\int_0^sN\log \phi(\tilde{X}_r)\,\vd l_r\r)
\end{align}
with
\begin{align}
|d P_t^Nu|(x)= \bigg|\mathbb{E}^x\l[u(\tilde{X}_t)\int_0^t h'(s)\tilde{\mathcal{Q}}_s dB_s\r]\bigg|\leq \|u\|_{\infty}\,\E^x\l[\l|\int_0^t h'(s)\tilde{\mathcal{Q}}_s dB_s\r|\r]
\end{align}
for any $h\in C^1([0,t])$ with $h(0)=0$ and $h(t)=1$. Now, using the fact that
\begin{equation}
\E^x\l[\int_0^t \l|h'(s)\r|^2 \|\tilde{\mathcal{Q}}_s\|^2 \vd s\r]^{1/2} \leq \|\phi\|_\infty \E^x\l[\int_0^t  \phi^{-2}(\tilde{X}_s) \l|h'(s)\r|^2 \|\tilde{\mathcal{Q}}_s\|^2 \vd s\r]^{1/2}
\end{equation}
it follows, as in the proof of Theorem \ref{global-esti}, that
\begin{equation}
|d P_t^Nu|(x) \leq\|\phi\|_\infty \|u\|_{\infty} \sqrt{\frac{2}{\pi}} \left(\int_0^t \l|h'(s)\r|^2 \E^x\l[\phi^{-2}(\tilde{X}_s)  \|\tilde{\mathcal{Q}}_s\|^2\r] \vd s\right)^{1/2}.
\end{equation}
To estimate the expectation, we see by the It\^{o} formula that
\begin{align*}
\vd \phi^{-2}(\tilde{X}_t)&=\langle \nabla \phi^{-2}(\tilde{X}_t), u_t\,\vd B_t \rangle+L\phi^{-2}(\tilde{X}_t)\,\vd t +N\phi^{-2}(\tilde{X}_t)\,\vd l_t\\
&=\langle \nabla \phi^{-2}(\tilde{X}_t), u_t\,\vd B_t \rangle-2\phi^{-2}(\tilde{X}_t)\l(-\frac{1}{2}\phi^{2}L\phi^{-2}(\tilde{X}_t)\,\vd t +N\log\phi(\tilde{X}_t)\,\vd l_t\r)
\end{align*}
which implies
\begin{align*}
\phi^{-2}(\tilde{X}_t)\exp\l(-\int_0^t\phi^{2}(\tilde{X}_s)L\phi^{-2}(\tilde{X}_s)\,\vd s+2\int_0^tN\log \phi(\tilde{X}_s)\,\vd l_s\r)
\end{align*}
is a local martingale, from which it follows that
\begin{align*}
\E\l[\phi^{-2}(\tilde{X}_t)\exp\left(-\int_0^t\phi^{2}(\tilde{X}_s)L\phi^{-2}(\tilde{X}_s)\,\vd s+2\int_0^tN\log \phi(\tilde{X}_s)\,\vd l_s\right)\r]\leq \phi^{-2}(x).
\end{align*}
Putting all this together, and using the fact that $\inf \phi =1$, we have
\begin{equation}
|d P_t^Nu|(x) \leq \|\phi\|_\infty \|u\|_{\infty} \sqrt{\frac{2}{\pi}} \left(\int_0^t \l|h'(s)\r|^2 \e^{-K_\phi s} \vd s\right)^{1/2}.
\end{equation}
Choosing $h$ as in Theorem \ref{global-esti}, with $K_\phi$ in place of $K_Z$, completes the proof. \qed
\end{proof}

Using information about the boundary, an explicit function $\phi$ can be constructed. For instance, following Wang's construction (see \cite[p.1436]{W07} or \cite[Theorem 3.2.9]{Wbook2}), we have the following corollary:

\begin{corollary}\label{cor-est}
Assume $\Ric^Z\geq K_Z$ for some constant $K_Z$ and that there exist non-negative constants $\sigma$ and $\theta$ such that $-\sigma\leq \II \leq \theta$ and a positive constant $r_0$ such that on $\partial_{r_0}M:=\{x\in M\colon \rho_{\partial}(x)\leq r_0\}$ the function $\rho_{\partial}$ is smooth, the norm of $Z$ is bounded and $\Sect \leq k$ for some $k\geq 0$. Then for $u\in \mathcal{B}_b(M)$ we have
\begin{align*}
\|dP_t^Nu\|_{\infty}\leq \sqrt{\frac{2}{\pi s}}\exp \l(\frac 12\sigma d r_1+\frac 12\l(K_Z-2\sigma \delta_{r_0}(Z)-\frac{2\sigma d}{r_1}-2\sigma^2\r)^-s\r) \|u\|_{\infty}
\end{align*}
for all $0<s\leq t$ where
\begin{equation}
r_1:=r_0\wedge \l(\frac{1}{\sqrt{k}}\arcsin\l(\sqrt{\frac{k}{k+\theta^2}}\r)\r)\quad\mbox{and}\quad \delta_{r_0}(Z):=\sup_{x\in \partial_{r_0} M}|Z|(x).
\end{equation}
In particular
\begin{align*}
\|d P_t^Nu\|_{\infty}\leq \sqrt{\frac{2\exp\left(\sigma d r_1+1\right)}{\pi(t \wedge 1)}}\,\max\l\{\l(K_Z-\sigma \delta_{r_0}(Z)-\frac{\sigma d}{r_1}-2\sigma^2\r)^-, 1\r\}^{1/2} \|u\|_{\infty}
\end{align*}
for all $t>0$.
\end{corollary}

\begin{proof}
Under the assumptions, we can construct the following function $\phi \in \mathcal{D}$:
\begin{align*}
\log \phi(x) =\frac{\sigma}{\alpha}\int_0^{\rho_{\partial}(x)}\big(\ell(s)-\ell(r_1)\big)^{1-d}\vd s \int_{s\wedge r_1}^{r_1}\big(\ell(u)-\ell(r_1)\big)^{d-1}\vd u,
\end{align*}
where
\begin{align}\label{fun-h}
\ell(t):=\cos{\sqrt{k}t}-\frac{\theta}{\sqrt{k}}\sin {\sqrt{k}t}, \quad t\geq 0,
\end{align}
so that $r_1=r_0\wedge \ell^{-1}(0)$, and
\begin{equation}
\alpha:=(1-\ell(r_1))^{1-d}\int_0^{r_1}\big(\ell(s)-\ell(r_1)\big)^{d-1}\vd s.
\end{equation}
Then, as checked in \cite{Ch-Tha-Tho:2017b} and \cite[Theorem 1.1]{Wang05}, we have
\begin{equation}
{\rm Ric}^Z+2L \log \phi-2|\nabla \log \phi|^2 \geq K_Z-2\sigma\delta_{r_0}(Z)-\frac{2 \sigma d}{r_1}-2\sigma^2
\end{equation}
and $\|\phi\|_{\infty}\leq \e^{\sigma d r_1/2}$. By Theorem \ref{Gradient-Neumann} and Remark \ref{rem:sinzerot}, it follows that
\begin{align*}
\| d P^N_tu\|_{\infty}\leq \sqrt{\frac{2 \e^{K_{\phi}^-s}}{\pi s}}\,\|u\|_{\infty}\,\|\phi\|_{\infty},\quad s\in (0,t].
\end{align*}
Taking
\begin{align*}
s=\left(1\vee\left(K_Z-2\sigma \delta_{r_0}(Z)-\frac{2\sigma d}{r_1}-2\sigma^2\right)^-\right)^{-1}\wedge t
\end{align*}
and using
\begin{equation}\label{eq:useest}
\left(t\wedge \frac{1}{c\vee 1}\right)^{-1/2}\leq \frac{\max \lbrace \sqrt{c},1\rbrace}{\sqrt{t\wedge 1}},
\end{equation}
which holds for any non-negative constant $c$, we complete the proof. \qed
\end{proof}

\subsection{Application: isoperimetric inequalities}\label{ss:isoineq}
Now suppose $L$ is of the form $L=\Delta+\nabla V$ for some $V\in C^2(M)$ and set $\mu(\vd x):=\e^{V(x)}\vd x$. Consider the following two isoperimetric constants:
\begin{align*}
\kappa^D:=\inf_{ \mu(A)>0}\frac{\mu_{\partial}(\partial A)}{\mu(A)},
\quad \kappa^N:=\inf_{ \mu(A)\in (0,\frac{1}{2}]}\frac{\mu_{\partial}(\partial A\setminus \partial M)}{\mu(A)},
\end{align*}
where $A$ runs over all smooth and connected bounded domains contained in $M$ and $\mu_{\partial}(\partial A)$ is the area of $\partial A$ induced by $\mu$. Consider also
\begin{align*}
&\lambda_1^D:=\inf\{\mu(|\nabla f|^2): f\in C_0^\infty(M),\ f|_{\partial M}=0,\ \mu(f^2)=1\},\\
&\lambda_1^N:=\inf\{\mu(|\nabla f|^2): f\in C_0^\infty(M),\ \mu(f^2)=1,\ \mu(f)=0\}.
\end{align*}
The quantities $\lambda_1^D$ and $\lambda_1^N$ are known as the first Dirichlet and Neumann eigenvalues, respectively. Note that since we do not assume $M$ compact, these quantities may not be true $L^2$-eigenvalues for the operator $L$. In general $\lambda_1^D>0$ (resp. $\lambda_1^N>0$) does not imply $\kappa^D>0$ (resp. $\kappa^N>0$), but these implications do hold under uniform gradient estimates for the corresponding diffusion semigroups. In particular, there is the following result, taken from \cite[Theorem 1.2]{Wang04} and \cite[Theorem~2.5.3]{Wbook1} (note that the semigroups considered in \cite{Wang04,Wbook1} have generator $L$, as opposed to $\frac{1}{2}L$):

\begin{theorem}[Wang \cite{Wang04}]\label{Poincare-iso}
Let $P_t^D$ and $P_t^N$ denote the Dirichlet and Neumann semigroups of $\frac{1}{2}L$ on $M$, respectively.
\begin{enumerate}
\item [$(1)$] If $\|d P_{2t}^Df\|_{\infty}\leq \frac{c}{\sqrt{t\wedge 1}}\|f\|_{\infty}$ holds for some $c>0$ and all $t>0$, $f\in \mathcal{B}_b(M)$, then
\begin{align*}
\kappa^D\geq \frac{1-\e^{-1}}{c}\left(\sqrt{\lambda_1^D}\wedge \lambda_1^D\right).
\end{align*}
\item [$(2)$] If $\mu(M)=1$ and $\|d P_{2t}^Nf\|_{\infty}\leq \frac{c}{\sqrt{t\wedge 1}}\|f\|_{\infty}$ holds for some $c>0$ and all
$t>0$, $f\in \mathcal{B}_b(M)$, then
\begin{align*}
\kappa^N\geq \frac{1-2\e^{-1}}{2c}\left(\sqrt{\lambda_1^N}\wedge \lambda_1^N\right).
\end{align*}
\end{enumerate}
\end{theorem}

For the case in which $Z = \nabla V$, for some function $V$, and setting ${\Ric}^V := \Ric - \Hess\, V$, Theorem \ref{Poincare-iso}, in conjunction with Theorem \ref{thm-gradient-Dirichlet} and Theorem \ref{Gradient-Neumann}, immediately implies the following two theorems:

\begin{theorem}
Suppose $\Ric^V\geq K_V$, $\Ric\geq K_0$ and $H_{\partial M}\geq \theta$ for some constants $K_V, K_0$ and~$\theta$. Then
\begin{align*}
\kappa^D\geq \frac{\sqrt{\pi}\,\left(\e^{-1}-\e^{-2}\right)\l(\sqrt{\lambda_1^D}\wedge \lambda_1^D\r)}{\max\left\{\sqrt{K_Z^-},1\right\}\left(1+\pi/8\right)+2\alpha_0\sqrt{\pi}}
\end{align*}
where $\alpha_0$ is defined as in \eqref{eq:alpha0defn}.
\end{theorem}

\begin{theorem}
Suppose ${\rm Ric}^V+2L \log \phi- 2|\nabla \log \phi|^2 \geq K_{\phi}$ for some $\phi\in \mathcal{D}$ and constant $K_{\phi}$. Then
\begin{align*}
\kappa^N\geq \frac{\sqrt{\pi}\,\left(\e^{-1}-2\e^{-2}\right)\l(\sqrt{\lambda_1}\wedge \lambda_1\r)}{2\max\left\{\sqrt{K_{\phi}^-},1\right\}\|\phi\|_{\infty}}.
\end{align*}
\end{theorem}

\begin{proof}
By Theorem \ref{Gradient-Neumann} and Remark \ref{rem:sinzerot} we have that
\begin{align*}
|dP^N_{2t}u|(x)\leq \sqrt{\frac{2}{\pi}} \l(\frac{K_{\phi}}{\e^{2K_{\phi}s}-1}\r)^{1/2}\|\phi\|_{\infty}\,\|u\|_{\infty} \leq \frac{\e^{K_Z^- s}}{\sqrt{\pi s}}\,  \|\phi\|_{\infty}\,\|u\|_{\infty}
\end{align*}
for all $s \in (0,t]$. Choosing $s = t \wedge \frac{1}{1\vee K_Z^-}$ yields the result, by Theorem \ref{Poincare-iso} and inequality \eqref{eq:useest}. \qed
\end{proof}

\section{Gradient estimates for $C_b^2$ functions}\label{sec:quant}

In this section we apply the gradient estimates of the previous section to obtain uniform estimates for the derivatives of $C^2_b$-functions (that is, bounded twice continuously differentiable functions with bounded derivatives). We have three different cases, depending on the boundary behaviour of the function.

\subsection{No boundary}

The estimates of this subsection are uniform versions of the localized estimates that were recently proved by the authors in \cite{Ch-Tho-Tha:2017a}.

\begin{theorem}\label{thm:uniformests}
Suppose $\Ric^Z \geq K_Z$ for some constant $K_Z$. Then for all $u \in C_b^2(M)$ we have
\begin{align}
\ |d u|(x)\leq
\begin{cases}\displaystyle
\sqrt{\frac2{\pi }}\l(\l(\frac{K_Z}{\e^{K_Z t}-1}\r)^{1/2}\|u\|_\infty + \frac{1}{ \sqrt{-K_Z}} \log\left(\sqrt{\e^{-K_Z t}-1} + \e^{-K_Z t/2}\right)\|Lu\|_\infty\r),\\[3mm]
\displaystyle\sqrt{\frac2\pi}\l(\frac{1}{\sqrt{t}}\,\|u\|_\infty+\sqrt{t}\|Lu\|_\infty\r),\\[3mm]
\displaystyle\sqrt{\frac{2}{\pi}}\l(\l(\frac{K_Z}{\e^{K_Z t}-1}\r)^{1/2}\|u\|_\infty + \frac{1}{ \sqrt{K_Z}} \tan^{-1}\left(\sqrt{\e^{K_Z t}-1}\right)\|Lu\|_\infty\r),
\end{cases}
\end{align}
\normalsize
for the cases $K_Z<0$, $K_Z=0$ and $K_Z >0$, respectively, for all $t>0$.
\end{theorem}

\begin{proof}
As before, denote by $P_t$ the semigroup for the diffusion $X_t$, with generator $\frac{1}{2}L$. By differentiating the Kolmogorov equation, we get
\begin{align}\label{main-ineq-1}
|d u|(x)\leq |d P_tu|(x)+\frac{1}{2}\int_0^t|d P_sLu|(x)\,\vd s.
\end{align}
By Theorem \ref{global-esti} we have
\begin{align*}
|d P_tu|(x)\leq \sqrt{\frac2\pi}\l(\frac{K_Z}{\e^{K_Zt}-1}\r)^{1/2}\|u\|_{\infty}.
\end{align*}
Combining this with the first term on the right-hand side of \eqref{main-ineq-1}, and similarly for the second term, we find
\begin{equation}\label{eq:prelimestuni}
|d u|(x)\leq  \sqrt{\frac2\pi}\l(\l(\frac{K_Z}{\e^{K_Zt}-1}\r)^{1/2}\|u\|_\infty+ \frac{1}{2}\int_0^t \l(\frac{K_Z}{\e^{K_Z s}-1}\r)^{1/2}ds\,\|Lu\|_\infty\r).
\end{equation}
The right-hand side of this inequality is, by calculation, equal to the expressions given in the theorem, for each of the three cases. \qed
\end{proof}

Minimizing over $t$, we obtain the following corollary:

\begin{corollary}\label{cor:minnedests}
Suppose $\Ric^Z \geq K_Z$ for some constant $K_Z$ with $\|L u\|_{\infty}>0$. Then for all $u \in C_b^2(M)$ we have
\begin{align}
|d u|^2(x)\leq\text{ }&
\begin{cases}\displaystyle
\frac{2}{\pi}\|u\|_{\infty} \|L u\|_{\infty}\left( \sqrt{1+\beta} + \frac{\sinh^{-1}(\sqrt{\beta})}{\sqrt{\beta}}\right)^2,\\
\displaystyle\frac{8}{\pi}\|u\|_{\infty} \|L u\|_{\infty},\\
\displaystyle\frac{2}{\pi}\|u\|_{\infty} \|L u\|_{\infty}\l(\sqrt{1+\beta} + \frac{\tan^{-1}\left(\sqrt{\frac{-\beta}{1+\beta}}\right)}{\sqrt{-\beta}}\r)^2,\\
\displaystyle\sqrt{\frac{\pi}{2 K_Z}}\|L u\|_{\infty},\\
\end{cases}
\end{align}
for the cases $K_Z < 0$, $K_Z = 0$, $\|L u\|_{\infty} \|u\|_{\infty}^{-1} > K_Z > 0$ and $K_Z \geq \|L u\|_{\infty} \|u\|_{\infty}^{-1}$, respectively, where $\beta := -K_Z\|u\|_{\infty} \|L u\|_{\infty}^{-1} $.
\end{corollary}

Note that the right-hand side of the above inequality is continuous in $K_Z$. In particular
\begin{equation}
\lim_{\beta \downarrow 0}\left(\sqrt{1+\beta} + \frac{\sinh^{-1}(\sqrt{\beta})}{\sqrt{\beta}}\right)=2=\lim_{\beta \uparrow 0}\l(\sqrt{1+\beta} + \frac{\tan^{-1}\left(\sqrt{\frac{-\beta}{1+\beta}}\right)}{\sqrt{-\beta}}\r)
\end{equation}
and similarly for the two cases concerning $K_Z>0$. For the case $L u=0$ (which for $Z=0$ is to say that $u$ is harmonic), Theorem \ref{thm:uniformests} recovers the well-known fact that if $\Ric^Z \geq 0$ then such $u$ must be constant. But more generally, Theorem \ref{thm:uniformests} implies that if $\Ric^Z\geq K_Z$ with $K_Z\leq 0$, with $u$ a bounded $C^2$ function satisfying $Lu = 0$, then
\begin{equation}
\|du\|_\infty \leq \sqrt{\frac{-2K_Z}{\pi}}\,\|u\|_\infty.
\end{equation}

For a simpler estimate than the one given by Theorem \ref{thm:uniformests}, there is the following, in which we introduce a parameter $\delta$ to emphasise that there is no explicit dependence on time (see also Subsection \ref{ss:spec} below):

\begin{corollary}\label{cor-1}
Suppose $\Ric^Z\geq K_Z$ for some constant $K_Z$. Then for $u\in C_b^2(M)$ we have
\begin{align*}
\|du\|_{\infty}\leq \sqrt{\frac2\pi}\exp\left(\frac{K_Z^{-}}{2\delta^2}\right)\left(\delta \|u\|_{\infty}+\delta^{-1}\|Lu\|_{\infty}\right)
\end{align*}
for all $\delta >0$.
\end{corollary}

\begin{proof}
This follows from Theorem \ref{thm:uniformests} by \eqref{eq:prelimestuni} and the fact that
\begin{equation}
\l(\frac{K_Z}{\e^{K_Zt}-1}\r)^{1/2} \leq \exp\left(\frac{K^-_Z t}{2}\right),
\end{equation}
by setting $t = \delta^{-2}$. \qed
\end{proof}

\subsection{Dirichlet boundary}

\begin{theorem}\label{c1-est-d}
Suppose $\Ric^Z\geq K_Z$, $\Ric\geq K_0,$ and $H_{\partial M}\geq \theta$ for some constants $K_V, K_0$ and~$\theta$. Then for $u\in C_b^2(M)$ with $u|_{\partial M}=0$ we have
\begin{align*}
\|d u\|_{\infty}\leq \exp\left(\frac{K_Z^-}{2\delta^2}\right)\l(\sqrt{\frac{2}{\pi}}+\frac{1}{4}\sqrt{\frac{\pi}{2}}+\frac{2\alpha_0}{\delta}\r) \left(\delta \|u\|_{\infty}+\delta^{-1}\|Lu\|_{\infty}\right)
\end{align*}
for all $\delta >0$, where
\begin{align*}
\alpha_0=\frac{1}{2}\l(\max\l\{\theta^{-}, \sqrt{(d-1)K_0^{-}}\r\}+\|Z\|_{\infty}\r).
\end{align*}
\end{theorem}

\begin{proof}
By It\^{o}'s formula we have
\begin{equation}\label{eq:expito}
\E\left[ u(X_{t\wedge \tau}(x))\right] = u(x) + \frac{1}{2} \int_0^t \E \left[ 1_{\lbrace s < \tau \rbrace} (Lu)(X_s(x))\right] ds.
\end{equation}
Equation \eqref{eq:expito} can be rearranged as
\begin{equation}\label{eq:expito2}
u(x) = P^D_tu(x) - \frac{1}{2}\int_0^t P^D_s(Lu)(x) \,ds
\end{equation}
and so, by differentiating and applying Theorem \ref{thm-gradient-Dirichlet}, we obtain
\begin{align*}
|du|(x)\leq \e^{K_Z^-t/2}\l(\sqrt{\frac{2}{\pi}}+\frac{1}{4}\sqrt{\frac{\pi}{2}}+2\sqrt{t}\alpha_0\r) \l(\frac{1}{\sqrt{t}}\,\|u\|_{\infty}+\sqrt{t}\,\|Lu\|_{\infty}\r)
\end{align*}
which yields the estimate by setting $t=\delta^{-2}$. \qed
\end{proof}

\subsection{Neumann boundary}

\begin{theorem}\label{c1-est-n}
If there exist $\phi\in \mathcal{D}$ and a constant $K_{\phi}$ such that
\begin{align}
{\rm Ric}^Z+2L \log \phi- 2|\nabla \log \phi|^2 \geq K_{\phi}
\end{align}
then for $u\in C_b^2(M)$ such that $Nu|_{\partial M}=0$, we have
$$\|du\|_{\infty}\leq \sqrt{\frac{2}{\pi}}\exp\left(\frac{K_{\phi}^{-}}{2\delta^2}\right)\|\phi\|_{\infty}\left(\delta \|u\|_{\infty}+\delta^{-1}\|Lu\|_{\infty}\right)$$
for all $\delta >0$.
\end{theorem}

\begin{proof}
Recalling that $P_t^N$ is the Neumann semigroup with respect to the operator $\frac{1}{2}L$, by differentiating the Kolmogorov equation we have
\begin{align*}
|du|(x)\leq |dP_t^Nu|(x)+\frac{1}{2}\int_0^t|dP_s^NLu|(x) ds.
\end{align*}
By Theorem \ref{Gradient-Neumann}, which is a consequence of our recent result proved in \cite{Ch-Tha-Tho:2017b}, we know that
\begin{align*}
|dP^N_tu|(x)\leq \sqrt{\frac{2}{\pi}} \|\phi\|_{\infty} \l(\frac{K_{\phi}}{\e^{K_{\phi}t}-1}\r)^{1/2}\|u\|_{\infty} \leq \sqrt{\frac{2}{\pi}} {\frac{\e^{K_{\phi}^{-}t/2}}{\sqrt{t}}}\|\phi\|_{\infty} \|u\|_{\infty}
\end{align*}
and thus we obtain the result directly as before, in Corollary \ref{cor-1}. \qed
\end{proof}

Note that given such a $K_\phi$, estimates of the type given by Theorem \ref{thm:uniformests} are also available. It suffices to say that the estimates of Theorem \ref{thm:uniformests} and Corollary \ref{cor:minnedests} carry over to the Neumann setting, so long as one replaces the constant $K_Z$ by $K_\phi$ and remembers to include also the factor $\|\phi\|_\infty$. For explicit $\phi$, as explained in the proof of Corollary \ref{cor-est}, there is the following corollary of Theorem \ref{c1-est-n}:

\begin{corollary}\label{cor-c1-est-d}
Under the assumptions of Corollary \ref{cor-est}, for $u\in C_b^2(M)$ with $Nu|_{\partial M}=0$, we have
\begin{align*}
\|du\|_{\infty}&\leq \sqrt{\frac{2}{\pi}}\exp\left(\frac 1{2}\sigma d r_1+\frac 1{2\delta^2}\l(K_Z-2\sigma \delta_{r_0}(Z)-\frac{2\sigma d}{r_1}-2\sigma^2\r)^-\right)\\
&\qquad\times\left(\delta \|u\|_{\infty}+\delta^{-1}\|Lu\|_{\infty}\right)
\end{align*}
for all $\delta >0$, where $r_1$ and $\delta_{r_0}(Z)$ are defined as in Corollary \ref{cor-est}.
\end{corollary}

\subsection{Application: spectral projection operators}\label{ss:spec}
In this subsection, we first suppose that $M$ is a compact Riemannian manifold without boundary, of dimension $d$ as before. Let $0<\lambda_1\leq \lambda_2\leq \ldots$ denote the eigenvalues of $\Delta$ and let $\{e_j(x)\}$ be the associated real orthonormal basis of $L^2(M)$ consisting of eigenfunctions. For $f\in L^2(M)$, set
\begin{align*}
e_j(f)(x):=e_j(x)\int_Mf(y)e_j(y)dy
\end{align*}
and define the unit band spectral projection operators $\chi_{\lambda}$ by
\begin{align*}
\chi_{\lambda}f:=\sum_{\lambda_j\in [\lambda,\lambda+1)}e_j(f).
\end{align*}
The study of $L^p$-estimates for the such spectral projections has a long
history. For example, under the present assumptions, Sogge \cite{SS07,Sogge02,Sogge88} proved that there exists a constant $C>0$ such that
\begin{align}\label{Sogge-esti}
\|\chi_{\lambda}f\|_{p}\leq C \lambda^{\sigma(p)} \|f\|_2,\quad \lambda\geq 1,\ p \geq 2,
\end{align}
where
$$\sigma(p)=\max\l\{\frac{d-1}{2}-\frac{d}{p}, \frac{d-1}{2}\l(\frac{1}{2}-\frac{1}{p}\r)\r\}.$$
In particular, for $p=\infty$ and $\lambda\geq 1$, we have
\begin{align*}
\|\chi_{\lambda}f\|_{\infty}\leq C\lambda^{(d-1)/2}\,\|f\|_2.
\end{align*}
Moreover, from this and the Cauchy-Schwartz inequality, for each point $x\in M$, we find
\begin{align}\label{GE-chi}
|\Delta \chi_{\lambda}f|^2(x)&=\l(\sum_{\lambda_j\in [\lambda, \lambda+1)}(\lambda_j^2e_j(x))\int_Me_j(y)f(y)\,\vd y\r)^2 \notag\\
&\leq C(\lambda+1)^4\lambda^{n-1}\|\chi_{\lambda}f\|^2_2 \notag\\
&\leq C\lambda^{n+3}\|f\|^2_2,\quad \lambda\geq 1.
\end{align}
This leads us to the following theorem:

\begin{theorem}
Suppose $M$ is a compact manifold without boundary. Then there exists a constant $C>0$ such that
\begin{align*}
\|d\chi_{\lambda}f\|_{\infty}\leq C\lambda^{(n+1)/2}\|f\|_{2},\quad \lambda\geq 1.
\end{align*}
\end{theorem}

\begin{proof}
By Corollary \ref{cor-1}, we know that if $\Ric^Z\geq K_Z$ for some constant $K_Z$ then for $\lambda>1$, we have
\begin{align*}
\|d\chi_{\lambda}f\|_{\infty}\leq \sqrt{\frac{2}{\pi}}\e^{K_Z^-}\l(\lambda\|\chi_{\lambda}f\|_{\infty}+\lambda^{-1} \|\Delta\chi_{\lambda}f\|_{\infty}\r).
\end{align*}
Combining this with \eqref{Sogge-esti} and \eqref{GE-chi} completes the proof. \qed
\end{proof}

Now suppose that $M$ is a compact Riemannian manifold with boundary. Let $0<\lambda^D_1\leq \lambda^D_2\leq \ldots$ denote the corresponding Dirichlet eigenvalues with respect to $\Delta$. Let $\{e_j^D(x)\}$ be the associated orthonormal basis of eigenfunctions in $L^2(M)$. For $f\in L^2(M)$, define
\begin{align*}
e_j^D(f)(x)=e_j^D(x)\int_M f(y)e_j^D(y)dy
\end{align*}
and define the unit band spectral projection operator
\begin{align*}
\chi_{\lambda}^D(f)=\sum_{\lambda_j\in [\lambda,\lambda+1)}e^D_j(f).
\end{align*}
Let $0<\lambda_1^N\leq \lambda_2^N\leq \ldots$ denote the corresponding Neumann eigenvalues with respect to $\Delta$, and define the objects $\{e_j^N(x)\}$, $e_j^N(f)$ and $\chi_{\lambda}^N(f)$ analogously.

\begin{theorem}
Suppose $M$ is a compact manifold with boundary. Then there exist constants $C(D)$ and $C(N)$ such that
\begin{align*}
&\|d \chi^D_{\lambda}(f)\|_{\infty}\leq C(D)\lambda^{(n+1)/2}\|f\|_{2}, \quad \|d \chi^N_{\lambda}(f)\|_{\infty}\leq C(N)\lambda^{(n+1)/2}\|f\|_{2}
\end{align*}
for $\lambda \geq 1$.
\end{theorem}

\begin{proof}
By \cite{Sogge88} and \cite{Xu09}, we know that
\begin{align}\label{Est-chi-ND}
\|\chi_{\lambda} ^Df\|_{\infty}\leq C\lambda^{(n-1)/2}\|f\|_2,\quad  \|\chi_{\lambda}^N f\|_{\infty}\leq C\lambda^{(n-1)/2}\|f\|_2,\quad  \lambda \geq 1.
\end{align}
By the same argument as in \eqref{GE-chi}, we know that
\begin{align}\label{GE-chi-ND}
|\Delta \chi^D_{\lambda}f|^2(x)
&\leq C\lambda^{n+3}\|f\|^2_2, \quad |\Delta \chi^N_{\lambda}f|^2(x)\leq C\lambda^{n+3}\|f\|^2_2,\quad \lambda\geq 1.
\end{align}
Using Theorem \ref{c1-est-d}, with $K_V, K_0$ and $\theta$ constants chosen such that $\Ric^Z\geq K_Z$, $\Ric\geq K_0,$ and $H_{\partial M}\geq \theta$, for $\lambda\geq 1$, we have
\begin{align*}
&\|d \chi_{\lambda}^D(f)\|_{\infty}\\[2mm]
&\ \leq  \sqrt{\frac{2}{\pi}}\e^{K_Z^{-}/2} \l(\sqrt{\frac{2}{\pi}}+\frac{1}{4}\sqrt{\frac{\pi}{2}}+\frac{2\alpha_0}{\lambda}\r) \left(\lambda \|\chi_{\lambda}^D(f)\|_{\infty}+\lambda^{-1}\|\Delta \chi_{\lambda}^D(f)\|_{\infty}\right)
\end{align*}
where $\alpha_0$ is defined by \eqref{eq:alpha0defn}. Similarly, letting $\sigma, K_Z, r_0, r_1, \delta_{r_0}(Z)$ and $\theta$ be the constants as in Corollary \ref{cor-c1-est-d}, for $\lambda\geq 1$, we have
\begin{align*}
\|d \chi_{\lambda}^N(f)\|_{\infty}&\leq \sqrt{\frac{2}{\pi}}\exp\left({\frac 12\sigma d r_1+\frac 12\l(K_Z-2\sigma \delta_{r_0}(Z)-\frac{2\sigma d}{r_1}-2\sigma^2\r)^-}\right)\\
&\qquad\times\left(\lambda \|u\|_{\infty}+\lambda^{-1}\|Lu\|_{\infty}\right).
\end{align*}
Combining this with \eqref{Est-chi-ND} and \eqref{GE-chi-ND}, we complete the proof. \qed
\end{proof}

\begin{acknowledgements}
This work has been supported by the Fonds National de la Recherche Luxembourg (FNR) under the OPEN scheme (project GEOMREV O14/7628746). The first named author acknowledges support by NSFC (Grant No.~11501508) and Zhejiang Provincial Natural Science Foundation of China (Grant No. LQ16A010009).
\end{acknowledgements}

\end{document}